\documentclass{amsart}
\pdfoutput=1

\usepackage[mathscr]{eucal}
\usepackage{amssymb}
\usepackage[usenames,dvipsnames]{xcolor}
\usepackage{xspace}%
\usepackage{amsthm}
\usepackage{bbold}
\usepackage{enumerate}
\usepackage{array}
\usepackage{amsmath}
\usepackage{etoolbox}

\numberwithin{equation}{section}
\usepackage[all]{xy}
\SelectTips{cm}{}
\newdir{ >}{{}*!/-10pt/\dir{>}}

\usepackage{xr-hyper}

\usepackage[colorlinks=true,linkcolor={Brown},citecolor={Brown},urlcolor={Brown}]{hyperref}

\usepackage{cleveref}
\crefname{Thm}{Theorem}{Theorems}
\crefname{Rem}{Remark}{Remarks}
\crefname{Prop}{Proposition}{Propositions}

\swapnumbers 
\newtheorem{Cor}[equation]{Corollary}
\newtheorem{Lem}[equation]{Lemma}
\newtheorem{Prop}[equation]{Proposition}
\newtheorem{Thm}[equation]{Theorem}

\theoremstyle{remark}
\newtheorem{Def}[equation]{Definition}

\newtheorem{Exa}[equation]{Example}
\newtheorem{Rem}[equation]{Remark}
\newtheorem{Rec}[equation]{Recollection}
\newtheorem{Cons}[equation]{Construction}

\newcommand{\nc}{\newcommand}
\nc{\dmo}{\DeclareMathOperator}

\dmo{\Chain}{Ch}
\dmo{\CInvname}{\chi}
\dmo{\cone}{cone}
\dmo{\Der}{D}
\dmo{\DPerm}{DPerm}
\dmo{\Hm}{H}
\dmo{\Hom}{Hom}
\dmo{\Id}{Id}
\dmo{\id}{id}
\dmo{\Img}{Im}
\dmo{\Ind}{Ind}
\dmo{\Infl}{Infl}
\dmo{\Ker}{Ker}
\dmo{\Komp}{K}
\dmo{\Locname}{Loc}
\dmo{\modname}{mod}%
\dmo{\Mod}{Mod}
\dmo{\noeth}{noeth}
\dmo{\opname}{op}
\dmo{\perm}{perm}
\dmo{\Perm}{Perm}
\dmo{\Res}{Res}
\dmo{\rmH}{H}
\dmo{\sing}{sing}
\dmo{\smallb}{b}
\dmo{\smallperf}{perf}
\dmo{\Spec}{Spec}
\dmo{\StabD}{\mathsf{sing}}
\dmo{\thick}{thick}

\nc{\ACAT}[2]{\Komp\Inj_{\perm}(#1;#2)}
\nc{\ac}{\mathrm{ac}}
\nc{\adj}{\dashv}
\nc{\aka}{{a.\,k.\,a.}\ }
\nc{\barFplus}{\barF^{\scriptscriptstyle+}}
\nc{\barF}{\bar{\Permtomod}}
\nc{\bbF}{\mathbb{F}}
\nc{\bbZ}{\mathbb{Z}}
\nc{\cat}[1]{\mathscr{#1}}
\nc{\cA}{\cat{A}}
\nc{\cc}{\mathsf{c}}
\nc{\CInvGR}{\CInv{G}{\CR}}
\nc{\CInvHR}{\CInv{H}{\CR}}
\nc{\CInv}[2]{\CInvname^{#1}} 
\nc{\cJ}{\cat{J}}
\nc{\colim}{\mathop{\mathrm{colim}}}
\nc{\CRG}{\CR G}
\nc{\CR}{R}
\nc{\Dbs}{\Db^{\sing}}
\nc{\Db}{\Der_{\smallb}}
\nc{\Dperf}{\Der_{\smallperf}}
\nc{\DPermGR}{\DPerm(G;\CR)}
\nc{\DpermGR}{\Dperm{G}{\CR}}
\nc{\Dperm}[2]{\Der_{\perm}(#1;#2)}
\nc{\DRperf}{\Der_{\kern-0.1em\CR\text{-}\kern-0.1em\smallperf}(\CRG)}
\nc{\Dsing}{\Der^{\sing}}
\nc{\eg}{{\sl e.g.}\@\xspace}
\nc{\Fplus}{\Permtomod^{\scriptscriptstyle+}}
\nc{\Fp}{\bbF_{\!p}}
\nc{\hook}{\hookrightarrow}
\nc{\ideal}[1]{\langle #1\rangle}
\nc{\ie}{{\sl i.e.}\@\xspace}
\nc{\ihom}{{\mathsf{hom}}} 
\nc{\Inj}{\mathrm{Inj}}
\nc{\into}{\mathop{\rightarrowtail}}
\nc{\isoto}{\overset{\sim}{\,\to\,}}
\nc{\Kac}{\Komp_{\mathrm{ac}}}
\nc{\Kbac}{\Komp_{\smallb,\mathrm{ac}}}
\nc{\Kb}{\Komp_{\smallb}}
\nc{\kk}{k}
\nc{\loccit}{{\sl loc.\ cit.}\xspace}
\nc{\Loc}[1]{\Locname(#1)}
\nc{\Lotimes}{\otimes^{\rmL}}
\nc{\lto}{\leftarrow}
\nc{\Mid}{\,\big|\,}
\nc{\MMod}[1]{\Mod(#1)}%
\nc{\mmod}[1]{\modname(#1)}%
\nc{\onto}{\mathop{\twoheadrightarrow}}
\nc{\op}{^{\opname}}
\nc{\permGR}{\perm(G;\CR)}
\nc{\Permtomod}{\Upsilon}
\nc{\ppermutation}{$\natural$-permutation\xspace}
\nc{\qquadtext}[1]{\qquad\textrm{#1}\qquad}
\nc{\restr}[1]{_{|_{\scriptstyle #1}}}
\nc{\rmh}{\mathsf{h}}
\nc{\rmL}{\mathsf{L}}
\nc{\rmR}{\mathsf{R}} 
\nc{\SET}[2]{\big\{\,#1\Mid#2\,\big\}}
\nc{\too}{\mathop{\longrightarrow}\limits}
\nc{\To}{\Rightarrow}
\nc{\unit}{\mathbb{1}}
\nc{\xinto}[1]{\overset{#1}{\,\into\,}}
\nc{\xonto}[1]{\overset{#1}{\,\onto\,}}
\nc{\xto}[1]{\xrightarrow{#1}}

\date{\today}
\author{Paul Balmer}
\address{Paul Balmer, UCLA Mathematics Department, Los Angeles, CA 90095-1555, USA}
\email{balmer@math.ucla.edu}
\urladdr{https://www.math.ucla.edu/~balmer}

\author{Martin Gallauer}
\address{Martin Gallauer, Max-Planck-Institut f\"ur Mathematik, 53111 Bonn, Germany}
\email{gallauer@mpim-bonn.mpg.de}
\urladdr{https://guests.mpim-bonn.mpg.de/gallauer}

\hypersetup{pdfauthor={Paul Balmer and Martin Gallauer},pdftitle={Permutation modules and cohomological singularity}}

\begin{document}


\title[permutation modules and cohomological singularity]{permutation modules and\\cohomological singularity}

\begin{abstract}
We define a new invariant of finitely generated representations of a finite group, with coefficients in a commutative noetherian ring.
This invariant uses group cohomology and takes values in the singularity category of the coefficient ring.
It detects which representations are controlled by permutation modules.
\end{abstract}

\subjclass[2020]{}

\thanks{First-named author supported by NSF grant~DMS-1901696.
The authors would like to thank the Isaac Newton Institute for Mathematical Sciences for support and hospitality during the programme \textit{K-theory, algebraic cycles and motivic homotopy theory} when work on this paper was undertaken. This programme was supported by EPSRC grant number EP/R014604/1.}

\maketitle

\begin{flushright}
Dedicated to Henning Krause \\ on the occasion of his 60th birthday
\end{flushright}
\bigbreak

\tableofcontents

\section{Introduction}
\label{sec:intro}%


\emph{In the whole paper, $G$ is a finite group and $\CR$ is a commutative noetherian ring.}

\medbreak

Let $M$ be a finitely generated $\CRG$-module.
In this article we define an invariant which measures how singular the cohomology of~$M$ is.
It allows us to conclude the theme of~\cite{balmer-gallauer:resol-small}, where we started exploring how much of the $\CR$-linear representation theory of~$G$ is controlled by \emph{permutation modules} (\Cref{Rec:perm}).
One motivation is that general $\CRG$-modules are typically wild, whereas permutation ones are much simpler.
Here we prove a precise version of the following slogan:
\begin{center}
The $\CRG$-module $M$ is controlled by permutation modules \\ if and only if its cohomology is not singular.
\end{center}
In the remainder of the introduction we describe the invariant and make this statement precise.

\medbreak
It will be convenient to view $M$ as an object of $\Db(\CRG)$, the bounded derived category of finitely generated $\CRG$-modules.
We consider the thick triangulated subcategory of~$\Db(\CRG)$ generated by finitely generated permutation modules
\begin{equation}
\label{eq:Dperm}%
\DpermGR=\thick\left\{\CR(G/H)\mid H\le G\right\}
\end{equation}
as the part of~$\Db(\CRG)$ that is `controlled by permutation modules'.
This interpretation is justified by our result in~\cite{balmer-gallauer:resol-small} that the
canonical functor which sends a complex of permutation modules to itself viewed in the derived category
\begin{equation}
\label{eq:Kb->Db}%
\Permtomod\colon \Kb(\perm(G;\CR))\too \Db(\CRG)
\end{equation}
is essentially a localization onto this~$\DpermGR$.
More precisely, $\Permtomod$ induces, after quotienting-out its kernel and idempotent-completing~$(\ldots)^{\natural}$, an equivalence
\begin{equation}
\label{eq:Kb/Kbac=Dperm}%
\bar{\Permtomod}\colon\left(\frac{\Kb(\perm(G;\CR))}{\Kbac(\perm(G;\CR))}\right)^\natural\xto{\sim}\DpermGR.
\end{equation}
The announced invariant will be a functor defined on $\Db(\CRG)$ which vanishes exactly on $\DpermGR$.

To define it, recall the \emph{(small) singularity category}~\cite{orlov:singularities}
\[
\Dbs(\CR)=\frac{\Db(\CR)}{\Dperf(\CR)}
\]
which measures how far the ring~$\CR$ is from being regular. See also Stevenson~\cite{stevenson:singularity}.
Since the cohomology of~$M$ is typically unbounded, we will also require the \emph{`big' singularity category} $\Dsing(\CR)$, following Krause~\cite{krause:stabX}.
It is a compactly-generated triangulated category, whose subcategory of compact objects coincides with the idempotent-completion of the above~$\Dbs(\CR)$.
Krause extends the evident quotient functor $\Db(\CR)\onto \Dbs(\CR)$ to a functor defined on unbounded complexes of arbitrary modules.
We call this extension the \emph{singularity functor}
\[
\StabD\colon \Der(\CR)=\Der(\MMod{\CR})\too \Dsing(\CR).
\]

\smallskip

For each subgroup~$H\le G$, let $(-)^{\rmh H}\colon\Der(\CRG)\to \Der(\CR)$ be the right-derived functor of the $H$-fixed-points functor $(-)^H$.
We can now state our main result.
\begin{Thm}[\Cref{Thm:main}]
\label{Thm:main-intro}%
The subcategory $\DpermGR$ of~$\Db(\CRG)$,\break given in~\eqref{eq:Dperm}, consists of those complexes~$X\in\Db(\CRG)$ such that the invariants
\begin{equation}
\label{eq:CInv-intro}%
\CInvHR(X):=\StabD(X^{\rmh H})
\end{equation}
vanish in the big singularity category~$\Dsing(\CR)$, for every subgroup~$H\le G$.
\end{Thm}

The functor $(-)^{\rmh G}\colon\Der(\CRG)\to \Der(\CR)$ is the derived-category version of ordinary group cohomology, that is, the following left-hand square commutes:
\begin{equation}
\label{eq:intro}%
\vcenter{\xymatrix@C=4em{
\MMod{\CRG}\ \ar@{^(->}[r] \ar[d]_-{\rmH^*(G,-)}
& \Der(\CRG) \ar[d]_-{(-)^{\rmh G}} \ar[rd]^-{\CInvGR}_-{\scriptscriptstyle\textrm{(def)}}
\\
\MMod{\CR}
& \Der(\CR) \ar[r]_-{\StabD} \ar[l]^-{\rmH^\ast\,(=\rmH_{-\ast})}
& \Dsing(\CR).\!
}}
\end{equation}
For every subgroup $H\le G$, we call the invariant appearing in~\eqref{eq:CInv-intro}
\begin{equation}
\label{eq:coh-sing-functor}
\CInvHR\colon \Der(\CRG)\xto{\Res^G_H}\Der(\CR H)\xto{(-)^{\rmh H}} \Der(\CR)\xto{\StabD} \Dsing(\CR)
\end{equation}
the \emph{$H$-cohomological singularity} functor. See \Cref{sec:inv}.

To apply \Cref{Thm:main-intro} to a complex~$X\in\Db(\CRG)$ whose underlying complex of $\CR$-modules $\Res^G_1X$ is already perfect it suffices to test $\CInvHR(X)=0$ for the Sylow subgroups~$H$ of~$G$, or alternatively for the (maximal) elementary abelian subgroups. (See~\Cref{Rec:ImF-stability}.) In particular, if $G$ is a $p$-group, there are two conditions for a complex~$X\in\Db(\CRG)$ to belong to $\DpermGR$: the na\"{i}ve $\Res^G_1(X)\in\Dperf(\CR)$ and the new $\CInvGR(X)=0$. See \Cref{sec:main}.

To appreciate the strength of \Cref{Thm:main-intro}, observe that for $\CR$ regular the condition $\CInvHR(X)=0$ is trivially true in~$\Dsing(\CR)=0$.
Thereby we recover the non-trivial fact that $\Permtomod$ in~\eqref{eq:Kb->Db} is surjective-up-to-summands when $\CR$ is regular.

\medskip
The article is organized as follows.
In \Cref{sec:background} we explain our conventions and recall the singularity functor.
In \Cref{sec:inv} we define the invariant~$\CInv{H}{}$ and prove that the objects~$X$ of $\DpermGR$ satisfy $\CInv{H}{}(X)=0$.
In \Cref{sec:main} we prove the converse, namely \Cref{Thm:main-intro}.


\section{Recollections and preparations}
\label{sec:background}%


We recall basic notation and other conventions, mostly following~\cite{balmer-gallauer:resol-small}. Then we remind the reader of the singularity category of a ring. Beyond this, Krause's recent monograph~\cite{krause:H-book} provides general background on Grothendieck categories and on representation theory of finite groups.

\subsection*{Conventions}

Unless specified, modules are left modules. We denote the category of $\Lambda$-modules by $\MMod{\Lambda}$ and its subcategory of finitely generated ones by~$\mmod{\Lambda}$.

Since fixed points~$(-)^H$ and other decorations (duals) appear in the exponent, we use homological notation for complexes $\cdots \to M_n\to M_{n-1}\to \cdots$.
We write $\Chain_?$, $\Komp_?$, and $\Der_?$ for, respectively, the category of chain complexes, its homotopy category and its derived category, with $?\in\{\emptyset,\mathrm{b},+,-\}$ indicating boundedness conditions, as usual.
We abbreviate $\Db(\Lambda)$ for $\Db(\mmod{\Lambda})$ and $\Der(\Lambda)$ for~$\Der(\MMod{\Lambda})$.
When we speak of a module as a complex, we mean it concentrated in degree zero.

All triangulated subcategories are implicitly assumed to be replete (closed under isomorphisms). We abbreviate `thick' for `triangulated and thick' (\ie closed under direct summands). A triangulated subcategory is called localizing if it is closed under coproducts.
We write $\thick(\cA)$ (respectively, $\Loc{\cA}$) for the smallest thick (respectively, localizing) subcategory containing~$\cA$.

For an additive category~$\cat{A}$, we denote by $\cA^\natural$ its idempotent-completion (\aka Karoubi envelope). Recall from~\cite{balmer-schlichting:idempotent-completion} that $\Kb(\cA^\natural)\cong\Kb(\cA)^\natural$.
\begin{Rec}
\label{Rec:perm}%
If $A$ is a left $G$-set, we denote by $\CR(A)$ the free $\CR$-module with $G$-action extended $\CR$-linearly from its basis~$A$. An $\CRG$-module is called a \emph{permutation module} if it is isomorphic to $\CR(A)$ for some $G$-set~$A$. The additive category of permutation modules is denoted by $\Perm(G;\CR)$ and its subcategory of finitely generated ones by~$\perm(G;\CR)$.
\end{Rec}

\begin{Rec}
\label{Rec:tensor-R}%
\label{Rec:scalar-extension}%
We tensor $\CRG$-modules over~$\CR$ and use diagonal $G$-action:
\[
-\otimes_{\CR}-\colon\MMod{\CRG}\times\MMod{\CRG}\too\MMod{\CRG\otimes_{\CR}\CRG}\too\MMod{\CRG}.
\]
This tensor is  right-exact in each argument and can be left-derived as usual:
\[
-\Lotimes_{\CR}-\colon\Der_+(\CRG)\times\Der_+(\CRG)\too\Der_+(\CRG).
\]
If either $X$ or~$Y\in\Chain_+(\CRG)$ is degreewise $\CR$-flat, we have $X\Lotimes_{\CR}Y\cong X\otimes_{\CR}Y$.
We will also use the scalar-extension functor $\CR'\otimes_{\CR}-\colon\MMod{\CRG}\too\MMod{\CR' G}$
as well as its derived version~$\CR'\Lotimes_{\CR}-\colon\Der_+(\CRG)\to\Der_+(\CR' G)$.
It is easy to see that those functors preserve perfect complexes and complexes of permutation modules, and send $\CR$-perfect complexes to $\CR'$-perfect ones (see \Cref{Rec:R-perfect}).
\end{Rec}

\begin{Rec}
\label{Rec:image-Fbar}%
As mentioned in the introduction, the main object of~\cite{balmer-gallauer:resol-small} was the canonical tensor-triangulated functors $\Permtomod$ of~\eqref{eq:Kb->Db} and the induced functor
\[
\bar{\Permtomod}\colon\left(\frac{\Kb(\perm(G;\CR))}{\Kbac(\perm(G;\CR))}\right)^\natural\too\Db(\CRG).
\]
We proved~\cite[Theorem~4.3]{balmer-gallauer:resol-small} that $\bar{\Permtomod}$ is fully faithful, and consequently that its essential image is~$\DpermGR:=\Img(\bar{\Permtomod})=\thick(\perm(G;\CR))$ in~$\Db(\CRG)$ as in~\eqref{eq:Dperm}.
(\footnote{\,In~\cite{balmer-gallauer:resol-small}, this image $\Img\bar{\Permtomod}$ was denoted by both $\cat{P}(G;\CR)^\natural$ and $\cat{Q}(G;\CR)^\natural$ and we described its objects as those $X\in\Db(\CRG)$ such that $X\oplus\Sigma X$ admits `$m$-free permutation resolutions' for all $m\geq 0$.
However, in this paper we will not need this description.})
In other words, $\bar{\Permtomod}$ yields the tensor-triangulated equivalence~\eqref{eq:Kb/Kbac=Dperm}.
\end{Rec}

\begin{Rec}
\label{Rec:ImF-stability}%
It is easy to see that $\Dperm{-}{\CR}$ is stable under restriction and induction, and that $X\in\DpermGR$ if and only if $\Res^G_HX\in\Dperm{H}{\CR}$ for every Sylow subgroup $H\le G$. In fact, it suffices to test for $H\le G$ among the (maximal) elementary abelian subgroups of~$G$. Details can be found in~\cite[Proposition~2.20, Corollary~2.21 and Remark~4.9]{balmer-gallauer:resol-small}. However, the arguments in the present paper do \emph{not} use reduction to elementary abelian subgroups.
\end{Rec}

\begin{Rec}
\label{Rec:R-perfect}%
Recall~\cite[Definition~2.22]{balmer-gallauer:resol-small} that a complex $X\in\Db(\CRG)$ is \emph{$\CR$-perfect} if the underlying complex $\Res^G_1X\in\Db(\CR)$ is perfect.
We denote the thick tensor subcategory of $\CR$-perfect complexes by~$\DRperf$.
It is obvious that we always have $\DpermGR\subseteq\DRperf$. This is an equality if the order~$|G|$ of the group is invertible in~$\CR$; see~\cite[Proposition~2.20]{balmer-gallauer:resol-small}.

In summary, we have the following inclusions of small `derived' categories
\[
\Dperf(\CRG)\subseteq \DpermGR\subseteq\DRperf\subseteq\Db(\CRG).
\]
\end{Rec}

\subsection*{Singularity category}

The target of our `cohomological singularity' functor~\eqref{eq:coh-sing-functor} is the big singularity category of the coefficient ring~$\CR$. Let us remind the reader.

\begin{Rec}
\label{Rec:singularity}%
As in~\cite{krause:stabX}, let~$\cA$ be a locally noetherian Grothendieck category whose derived category is compactly-generated. For $\cA=\MMod{\CR}$, the subcategory of noetherian objects $\noeth\cA$ is $\mmod{\CR}$ and $\Der(\cA)$ is generated by~$\Der(\cA)^c=\Dperf(\CR)$. Similarly for~$\cA=\MMod{\CRG}$. The \emph{big singularity category} (or \emph{stable derived category}) of~$\cA$ is
\[
\Dsing(\cA)=\Kac(\Inj\cA)
\]
the full subcategory of the big homotopy category of injectives $\Komp(\Inj\cA)$ spanned by acyclic complexes.
There is a recollement $Q_{\lambda}\adj Q\adj Q_{\rho}$ and $I_{\lambda}\adj I\adj I_{\rho}$
\begin{equation}
\label{eq:sing-recollement}%
\vcenter{\xymatrix@R=2em{
\Der(\cA) \ar@<-1.5em>@{ >->}[d]_(.55){Q_{\lambda}} \ar@<1.5em>@{ >->}[d]^(.55){Q_{\rho}}
\\
\Komp(\Inj\cA) \ar@{->>}[u]|(.45){Q\vphantom{I^I_j}} \ar@<-1.5em>@{->>}[d]_(.45){I_{\lambda}} \ar@<1.5em>@{->>}[d]^(.45){I_{\rho}}
\\
\Dsing(\cA)=\Kac(\Inj\cA) \ar@{ >->}[u]|(.55){\vphantom{I^I_m}I}
}}
\end{equation}
for $I\colon \Kac(\Inj\cA)\into \Komp(\Inj\cA)$ the inclusion and $Q\colon\Komp(\Inj\cA)\onto\Der(\cA)$ the usual localization $Q^+\colon \Komp(\cA)\onto \Der(\cA)$ restricted to~$\Komp(\Inj\cA)$.
The \emph{singularity functor} (Krause's \emph{stabilization functor}) is defined as the composite
\[
\StabD_{\cA}:\Der(\cA)\xto{I_\lambda\circ Q_{\rho}}\Dsing(\cA).
\]
There is a natural transformation $Q_{\lambda}\to Q_{\rho}$ that is invertible on compacts:
\begin{equation}
\label{eq:Qr-Ql-iso}%
\qquad\qquad
Q_{\lambda}(X)\cong Q_{\rho}(X)
\qquad\textrm{if }X\in\Der(\cA)^c
\end{equation}
by \cite[Lemma~5.2]{krause:stabX}. On the larger subcategory $\Db(\noeth\cA)$, the right adjoint $Q_{\rho}$ defines an inverse to the equivalence of~\cite[\S\,2]{krause:stabX} identifying $\Komp(\Inj\cA)^c$
\[
Q\colon \Komp(\Inj\cA)^c\isoto \Db(\noeth\cA).
\]
In summary, we have a finite localization sequence $\Der(\cA)\ \xinto{Q_\lambda}\ \Komp(\Inj\cA)\ \xonto{I_\lambda}\ \Dsing(\cA)$ as in~\eqref{eq:sing-recollement}
and the triangulated category $\Dsing(\cA)$ is compactly-generated with compact part the (usual) \emph{small singularity category}, idempotent-completed:
\[
\Dsing(\cA)^c\cong\left(\frac{\Komp(\Inj\cA)^c}{(Q_{\lambda}\Der(\cA))^c}\right)^\natural\cong\left(\frac{\Db(\noeth\cA)}{\Der(\cA)^c}\right)^\natural=\Dbs(\cA)^\natural.
\]
Finally, we note that since $\Komp(\Inj\cA)$ is compactly-generated and since the inclusion $\Komp(\Inj\cA)\into\Komp(\cA)$, that we denote~$J$, preserves products and coproducts (because $\cA$ is locally noetherian), there is another useful triple of adjoints $J_{\lambda}\adj J\adj J_{\rho}$:
\begin{equation}
\label{eq:J-adjoints}%
\vcenter{\xymatrix@R=2em{
\Komp(\cA) \ar@<-1.5em>@{->>}[d]_(.45){J_{\lambda}} \ar@<1.5em>@{->>}[d]^(.45){J_{\rho}}
\\
\Komp(\Inj\cA). \ar@{ >->}[u]|(.55){\vphantom{I^I_m}J}
}}
\end{equation}
\end{Rec}

\begin{Lem}
\label{Lem:resolutions}%
Keep the above notation, \eg for $\cA=\MMod{\CRG}$. Let $X$ be in~$\Komp(\cA)$.
\begin{enumerate}[\rm(a)]
\item
\label{it:resolution-Qr}%
The object~$Q_{\rho}(X)\in\Komp(\Inj\cA)$ in~\eqref{eq:sing-recollement} is a K-injective resolution of~$X$. In particular, if $X\in\Komp_{-}(\cA)$ is left-bounded then $Q_{\rho}(X)$ belongs to~$\Komp_-(\Inj\cA)$.
\smallbreak
\item
\label{it:resolution-Jl}%
If $X\in\Komp_{-}(\cA)$ is left-bounded, the object~$J_{\lambda}(X)\in\Komp(\Inj\cA)$ in~\eqref{eq:J-adjoints} is an in\-jective resolution of~$X$. Hence if $X\in\Komp_{-,\ac}(\cA)$ is also acyclic then $J_{\lambda}(X)=0$.
\smallbreak
\item
\label{it:resolution-Qr=Jl}%
The restrictions of the two functors $Q_{\rho}$ and~$J_{\lambda}$ to $\Komp_-(\cA)$ are isomorphic.
\end{enumerate}
\end{Lem}

\begin{proof}
By \cite[Remark~3.7]{krause:stabX}, we have $Q\,J_{\lambda}\cong Q^{\scriptscriptstyle+}$, where $Q^{\scriptscriptstyle+}\colon \Komp(\cA)\onto \Der(\cA)$ is the (Bousfield) localization defining~$\Der(\cA)$. It follows that $J\,Q_{\rho}$ is right adjoint to~$Q^{\scriptscriptstyle+}$. So, if we let $i:=JQ_{\rho}\circ Q^{\scriptscriptstyle+}$, every $X\in\Komp(\cA)$ fits in an exact triangle
\begin{equation}
\label{eq:ai-triangle}%
a(X) \to X \xto{\eta} i(X) \to \Sigma a(X)
\end{equation}
in~$\Komp(\cA)$, where $a(X)\in\Ker(Q^{\scriptscriptstyle+})=\Kac(\cA)$ and $i(X)$ belongs to~$\Ker(Q^{\scriptscriptstyle+})^{\perp}=\Kac(\cA)^\perp$, that is, $i(X)$ is \emph{K-injective} by definition. In other words, \eqref{eq:ai-triangle} is the essentially unique triangle providing the K-injective resolution of~$X$ (see \cite[Corollary~3.9]{krause:stabX} if necessary). Suppressing the functors~$J$ and~$Q^{\scriptscriptstyle+}$ that are just the identity on objects, we have $i(X)=Q_{\rho}(X)$, which gives~\eqref{it:resolution-Qr}.

Let now $A\in\Komp_{-,\ac}(\cA)$. The unit $\eta'\colon A\to JJ_\lambda(A)$ is a map from a left-bounded acyclic to a complex of injectives, hence $\eta'=0$ in~$\Komp(\cA)$. But $J_{\lambda}(\eta')\colon J_{\lambda}\isoto J_\lambda JJ_{\lambda}$ is invertible ($J$ being fully faithful). Thus $J_{\lambda}(A)=0$, as in the second claim of~\eqref{it:resolution-Jl}.

Take now~$X\in\Komp_{-}(\cA)$ arbitrary and an injective resolution $i(X)\in\Komp_-(\Inj\cA)$. There is a triangle~\eqref{eq:ai-triangle} with $a(X)$ acyclic and left-bounded since~$X$ and~$i(X)$ are. By the above for $A=a(X)$, we already know that $J_{\lambda}(a(X))=0$. Applying~$J_{\lambda}$ to the triangle~\eqref{eq:ai-triangle} in question we get $J_{\lambda}(X)\cong J_{\lambda}(i(X))\cong i(X)$ since $i(X)\in\Komp(\Inj\cA)$. Hence~\eqref{it:resolution-Jl}.

Part~\eqref{it:resolution-Qr=Jl} is now immediate from the uniqueness of K-injective resolutions.
\end{proof}

Let us now specialize to~$\cA=\MMod{\CR}$.
\begin{Lem}
\label{Lem:Ker-sing}%
Let~$X$ be an object of~$\Der(\CR)$. The following are equivalent:
\begin{enumerate}[\rm(i)]
\item
\label{it:Ker-sing-i}%
The image of~$X$ in~$\Dsing(\CR)$ is zero: $\StabD(X)=0$.
\item
\label{it:Ker-sing-ii}%
$Q_{\rho}(X)$ belongs to the localizing subcategory of $\Komp(\Inj(\CR))$ generated by $Q_{\rho}(\CR)$.
\item
\label{it:Ker-sing-iii}%
For every $Y\in\Db(\CR)$, every map $Y\to X$ in~$\Der(\CR)$ factors via a perfect complex.
\end{enumerate}
\end{Lem}

\begin{proof}
We have $\StabD=I_{\lambda}\circ Q_{\rho}$ by definition. So we have $\StabD(X)=0$ if and only if $Q_{\rho}(X)\in\Ker(I_{\lambda})$ and by~\eqref{eq:sing-recollement} that kernel is
$\Ker(I_{\lambda})=Q_{\lambda}(\Der(\CR))=Q_{\lambda}(\Loc{\CR})=\Loc{Q_{\lambda}(\CR)}$, where the last equality holds since $Q_{\lambda}$ is coproduct-preserving and fully faithful. We get the formulation~\eqref{it:Ker-sing-ii} from $Q_{\lambda}(\CR)\cong Q_{\rho}(\CR)$ since~$\CR\in\Der(\CR)^c$; see \eqref{eq:Qr-Ql-iso}. To reformulate this as~\eqref{it:Ker-sing-iii}, recall from~\cite[\S\,2]{neeman:thomason-localization} that in a compactly-generated triangulated category, the localizing subcategory generated by a thick subcategory~$\cJ$ of compacts consists of those $X$ such that every map from the generators to~$X$ factors via an object of~$\cJ$. If we apply this to $\Komp(\Inj(\CR))$ and the object~$Q_{\rho}(X)$, we see that $Q_{\rho}(X)\in\Loc{Q_{\lambda}(\Dperf(\CR))}=\Loc{Q_{\rho}(\Dperf(\CR))}$ if and only if for every $Y\in\Db(\CR)$ every map~$Q_{\rho}(Y)\to Q_{\rho}(X)$ factors via $Q_{\rho}(P)$ for some $P\in\Dperf(\CR)$. This is equivalent to~\eqref{it:Ker-sing-iii} since $Q_{\rho}$ is fully faithful.
\end{proof}

\begin{Rem}
In fact, $\StabD(X)=0$ is also equivalent to $Q_{\lambda}(X)\isoto Q_{\rho}(X)$ but we will not need this in the sequel.
\end{Rem}


\section{Cohomological singularity}
\label{sec:inv}


In this section, we define the announced cohomological singularity functor~\eqref{eq:CInv-intro}.

\begin{Rec}
\label{Rec:cc}%
The functor that equips every $\CR$-module with trivial $G$-action
\[
\Infl_1^G\cong\ihom_{\CR}(\CR,-)\cong \CR\otimes_{\CR}-
\colon\MMod{\CR}\to \MMod{\CRG}
\]
has adjoints the usual $G$-orbits $(-)_G$ and $G$-fixed points $(-)^G$
\begin{equation}
\label{eq:Infl-adjoints}%
\vcenter{\xymatrix@R=1.8em{
\MMod{\CRG} \ar@<-2em>@/_2em/[d]_-{\CR\otimes_{\CRG}-\,=\,(-)_G} \ar@<2em>@/^2em/[d]^-{\ihom_{\CRG}(\CR,-)\,=\,(-)^G}
\\
\MMod{\CR}\ar[u]|-{\Infl_1^G}
}}
\end{equation}
This triple of adjoints passes to homotopy categories of complexes on the nose. For derived categories, we left-derive the left adjoint and right-derive the right one:
\begin{equation}
\kern4.5em\vcenter{\xymatrix@R=1.8em{
\Der(\CRG) \ar@<-2em>@/_2em/[d]_-{\CR\Lotimes_{\CRG}-=:(-)_{\rmh G}} \ar@<2em>@/^2em/[d]^-{\rmR\ihom_{\CRG}(\CR,-)=:(-)^{\rmh G}}
\\
\Der(\CR) \ar[u]|-{\Infl_1^G}
}}
\end{equation}
So $(-)^{\rmh G}$ provides a complex whose homology groups are $G$-cohomology as in~\eqref{eq:intro}.
\end{Rec}

\begin{Def}
\label{Def:coh-inv}%
Let $H\le G$ be a subgroup. The \emph{$H$-cohomological singularity} functor~$\CInvHR=\StabD_{\CR}\circ(-)^{\rmh H}$ is the following composite (see \Cref{Rec:singularity} for~$\StabD$):
\[
\CInvHR\colon\Der(\CRG)\xto{\Res^G_H}\Der(\CR H)\xrightarrow{(-)^{\rmh H}}\Der(\CR)\xto{\StabD_{\CR}}\Dsing(\CR).
\]
We say that a complex $X\in\Der(\CRG)$ is \emph{$H$-cohomologically perfect} if $\CInvHR(X)=0$.
We say that $X$ is \emph{cohomologically perfect}, if it is $H$-cohomologically perfect for all subgroups $H\le G$, that is, if~$\oplus_{H}\,\CInvHR(X)=0$ in~$\Dsing(\CR)$.
\end{Def}

\begin{Exa}
\label{Exa:R-perfect}%
For the trivial subgroup~$H=1$, a complex $X\in\Db(\CRG)$ is $1$-cohomologically perfect if and only if its underlying complex $\Res^G_1 X\in\Db(\CR)$ is perfect, that is, if and only if $X$ is \emph{$\CR$-perfect} in the sense of \Cref{Rec:R-perfect}.
\end{Exa}

\begin{Rem}
\label{Rem:R-perfect}%
We remind the reader that although $\Ker(\StabD)\cap \Db(\CR)=\Dperf(\CR)$, the kernel of~$\StabD\colon \Der(\CR)\to \Dsing(\CR)$ on the big derived category is larger than $\Dperf(\CR)$. For instance we will see in \Cref{Lem:chi-necessary} that $\CR^{\rmh G}$ belongs to that kernel.

Even when $H=1$, a big object $X\in \Der(\CRG)$ being $1$-cohomologically perfect is more flexible than being $\CR$-perfect although the two notions coincide when $X\in\Db(\CRG)$ is bounded, \ie when $\Res^G_1(X)\in\Db(\CR)$, as we saw in \Cref{Exa:R-perfect}.

For more general subgroups~$H\le G$, even a bounded complex $X\in\Db(\CRG)$ can be $H$-cohomologically perfect without $X^{\rmh H}$ being perfect; see \Cref{Exa:coh-perfect-not-1-perfect}.

We provide a further justification of the terminology in~\Cref{Rem:coh-ring-mod}.
\end{Rem}

\begin{Rem}
\label{Rem:balanced-c^!}%
The functor $(-)^G\cong\ihom_{\CRG}(\CR,-)$ is a special value of the bifunctor
\[
\MMod{\CRG}\op\times\MMod{\CRG}\xto{\ihom_{\CRG}(-,-)}\MMod{\CR}.
\]
It follows that for any $X\in\Der(\CRG)$ the object $X^{\rmh G}$ is represented by both
\begin{equation}
\label{eq:c!-explicit}%
\ihom_{\CRG}(P_\CR,X)
\qquadtext{and}
\ihom_{\CRG}(\CR,i(X))
\end{equation}
where $P_{\CR}\to \CR$ is a projective resolution of $\CR$ as an $\CRG$-module, and $X\to i(X)$ is a K-injective resolution of~$X$, for both are quasi-isomorphic to $\ihom_{\CRG}(P_\CR,i(X))$.
\end{Rem}

\begin{Rem}
\label{Rem:c^!-|G|-invert}%
If the order $|G|$ is invertible in~$\CR$, then the trivial $\CRG$-module $\CR$ is projective by Maschke. In that case, $(-)^G$ is exact and coincides with~$(-)^{\rmh G}$.
\end{Rem}

We can use this to see that being $G$-cohomologically perfect does not imply being $H$-cohomologically perfect for each subgroup~$H\le G$, even for $H=1$.
\begin{Exa}
\label{Exa:coh-perfect-not-1-perfect}%
Let $\CR=\bbZ/9$ and $G=C_2=\langle\, x\mid x^2=1\,\rangle$. Consider the $\CR$-module $M=\bbZ/3$ with the action of~$x$ by~$-1$. We have $M^G=0$ hence $M^{\rmh G}=0$ by \Cref{Rem:c^!-|G|-invert}. In particular, $M$ is $G$-cohomologically perfect but it is not $H$-cohomologically perfect for the subgroup~$H=1$ since $\bbZ/3$ is not perfect over~$\bbZ/9$.
\end{Exa}

Let us establish some generalities about the cohomological singularity functor.
\begin{Prop}
\label{Prop:CInv-Ind}%
Let $H\le G$ be a subgroup. There are canonical isomorphisms
\[
(-)^{\rmh G}\circ\Ind^G_H\cong(-)^{\rmh H}
\qquadtext{and}
\CInvGR\circ\Ind^G_H\cong\CInvHR.
\]
\end{Prop}

\begin{proof}
The first isomorphism follows from the relation $\Res^G_H\circ \Infl_1^G=\Infl_1^H$, by taking right adjoints and right-deriving. We use here that induction is also right-adjoint to restriction (because $[G\!:\!H]<\infty$) and is exact. See~\cite{krause:H-book} if necessary. The second isomorphism follows by post-composing with~$\StabD_{\CR}$.
\end{proof}

\begin{Cor}
\label{Cor:CInv-Ind-Res}%
Let $H\le G$. Then induction~$\Ind_H^G\colon \Db(\CR H)\to\Db(\CRG)$ and restriction $\Res^G_H\colon \Db(\CRG)\to\Db(\CR H)$ preserve cohomologically perfect complexes.
\end{Cor}

\begin{proof}
Restriction is built into \Cref{Def:coh-inv}. For induction, it follows immediately from the Mackey formula and \Cref{Prop:CInv-Ind}.
\end{proof}

Here is a key computation of our invariant~$\CInvGR$ of~\Cref{Def:coh-inv}.
\begin{Lem}
\label{Lem:chi-necessary}%
The object $\CInvGR(\CR)=\StabD(\CR^{\rmh G})$ is zero in~$\Dsing(\CR)$.
\end{Lem}

\begin{proof}
Recall from \Cref{Rem:balanced-c^!} that $\CR^{\rmh G}=\ihom_{\CRG}(P_{\CR},\CR)$ where $P_{\CR}$ is any projective resolution of~$\CR$ over~$\CRG$. Let $\Omega\CR=\Ker(\CRG\onto\CR)$ be the kernel of augmentation. By additivity, it suffices to prove that $\StabD(X)=0$ where $X=\ihom_{\CRG}(P,\CR)$, for any $\CRG$-projective resolution~$P$ of~$\CR\oplus\Omega\CR$.

By \cite[Corollary~5.3]{balmer-gallauer:resol-small}, there exists a sequence of quasi-isomorphisms of bounded complexes in~$\Chain_{\ge0}(\CRG)$
\[
\cdots \to Q(n+1)\to Q(n)\to \cdots \to Q(1)\to \CR\oplus\Omega\CR
\]
such that $Q(n)$ consists of finitely generated \ppermutation $\CRG$-modules (\ie direct summands of finitely generated permutation modules), and in the range $0 \le d < n$, the $\CRG$-module $Q(n)_d$ is projective and the map $Q(n+1)_d\to Q(n)_d$ is the identity. In particular, the above sequence of complexes $\cdots \to Q(n)\to \cdots \to Q(1)$ is eventually stationary in each degree and the limit $P=\lim_{n\to \infty} Q(n)$, computed degreewise, is a projective resolution of~$\CR\oplus\Omega\CR$.

Let us write for simplicity $(-)^\dag$ for $\ihom_{\CRG}(-,\CR)$. This additive functor induces degreewise the functor~$(-)^\dag=\ihom_{\CRG}(-,\CR)\colon \Komp(\CRG)\op\to \Komp(\CR)$. Our goal is to show that $\StabD(X)=0$ for $X=P^\dag$. Note that since $P=\lim_{n\to \infty}Q(n)$ in a degreewise stationary way, we also have $P^\dag=\colim_{n\to \infty}Q(n)^\dag$ in a degreewise stationary way, say, in $\Chain(\CR)$. The maps $Q(n)^\dag\to Q(n+1)^\dag\to\cdots\to P^\dag$ are the identity in degree~$d>-n$. Note also that $P^\dag\in \Komp_-(\CR)$ is left-bounded.

The key remark is that for every \ppermutation $\CRG$-module~$Q$, the $\CR$-module $Q^\dag$ is projective. Indeed, for $Q$ permutation, $Q^\dag$ is $\CR$-free. In our case, the complexes $Q(n)^\dag$ are therefore perfect over~$\CR$.

We prove $\StabD(P^\dag)=0$ via criterion~\eqref{it:Ker-sing-iii} in \Cref{Lem:Ker-sing}.
Let $Y\in\Db(\CR)$. A morphism $Y\to X=P^\dag$ in~$\Der(\CR)$ is given by a fraction $Y\lto L\to P^\dag$ in~$\Komp(\CR)$ where $L\to Y$ is a projective resolution of~$Y$, in particular $L\in \Komp_+(\CR)$ is right-bounded. It is then easy to see that any morphism $L\to P^\dag=\colim_{n\to \infty}Q(n)^\dag$ in~$\Komp(\CR)$ must factor via $Q(n)^\dag\to P^\dag$ for~$n\gg0$. Since $Q(n)^\dag$ is perfect, we have established condition~\eqref{it:Ker-sing-iii} of \Cref{Lem:Ker-sing} for our~$X$, giving us $\StabD(X)=0$ as wanted.
\end{proof}

Recall from~\eqref{eq:Dperm} the thick subcategory $\DpermGR$ of~$\Db(\CRG)$, generated by finitely generated permutation modules.

\begin{Prop}
\label{Prop:chi-necessary}%
Every object of $\DpermGR$ is cohomologically perfect.
\end{Prop}

\begin{proof}
As cohomologically perfect complexes form a thick subcategory of $\Db(\CRG)$, it suffices to show that $\CR(G/H)$ is cohomologically perfect for all \mbox{subgroups $H\le G$}.
The latter follows easily from \Cref{Lem:chi-necessary} and \Cref{Cor:CInv-Ind-Res}.
\end{proof}

We can now apply \Cref{Prop:chi-necessary} to show that being $\CR$-perfect (\Cref{Rec:R-perfect}) is not sufficient to belong to~$\DpermGR$.
\begin{Exa}
\label{Exa:1-perf-not-A(G)}%
Let $k=\bbF_2$ and consider the ring $\CR=k[x]/\ideal{x^2-1}$. Take $G=C_2=\langle\, y\mid y^2=1\,\rangle$ cyclic of order~2.
Let $M=\CR_x$ denote the ring $\CR$ viewed as an $\CRG$-module with the non-trivial action of~$y$ via~$x$.
This $M\in\Db(\CR C_2)$ is $\CR$-perfect but we claim that $\CInv{C_2}{\CR}(M)\neq 0$.
As $\CR C_2$ is self-injective, the following resolution
\[
0\to \CR_x\xto{y-x}\CR C_2\xto{y-x}\CR C_2\to\cdots
\]
is an injective resolution of~$\CR_x$. Computing $(\CR_x)^{\rmh G}$ in~$\Der(\CR)$ as in \eqref{eq:c!-explicit} with the above $i(M)=\cdots 0\to \CR C_2\xto{y-x}\CR C_2\xto{y-x}\CR C_2\to\cdots$, we get that $(\CR_x)^{\rmh G}$ is
\[
\cdots 0\to \CR\xto{1-x}\CR\xto{1-x}\CR\to\cdots
\]
and we deduce that $(\CR_x)^{\rmh G}\simeq k$ in~$\Der(\CR)$.
But $k\in\Db(\CR)$ is not perfect, hence $\CInvGR(M)\simeq\StabD(k)\neq 0$.
In other words, $M$ is not cohomologically perfect.
Using \Cref{Prop:chi-necessary}, this means $M\notin\DpermGR$.
\end{Exa}


\section{Main result}
\label{sec:main}


We saw in \Cref{Prop:chi-necessary} that $\DpermGR$ is contained in the subcategory of cohomologically perfect complexes (\Cref{Def:coh-inv}). Our goal in this section is to prove the reverse inclusion.
Two ideas will be key: the ``cohomology'' comonad $\Infl_1^G\circ(-)^{\rmh G}$ on cohomologically perfect objects, and compactness arguments. To make both work at the same time, we lift that comonad to the homotopy category of injectives,~$\Komp(\Inj(\CRG))$, whose compact part is the bounded derived category. The proof of our main result being somewhat long, we prove several shorter lemmas.
Let us first set the notation.

\begin{Cons}
\label{Cons:Krause}%
We are going to assemble the following diagram via~\cite[\S\,6]{krause:stabX}
\begin{equation}
\label{eq:main}%
\vcenter{\xymatrix@H=1em@R=1em{
& \kern2em \Komp(\Inj(\CRG))
\ar@<1em>[dd]^-{\displaystyle\hat \cc^!}
\ar@{->>}[ld]|-{Q}
\\
\Der(\CRG)\kern1em
\ar@<.7em>@{ >->}[ru]^-{Q_\lambda}
\ar@<-.7em>@{ >->}[ru]_-{Q_{\rho}}
\ar[dd]^(.47){\displaystyle\cc^!}
\\
& \kern1.5em \Komp(\Inj(\CR))
\ar[uu]^(.482){\displaystyle\hat \cc_*}
\ar@{->>}[ld]|-{Q}
\\
\Der(\CR)\kern1em
\ar@<.7em>@{ >->}[ru]^-{Q_\lambda}
\ar@<-.7em>@{ >->}[ru]_-{Q_{\rho}}
\ar@<1em>[uu]^-{\displaystyle\cc_*}
}}
\end{equation}

We already encountered the slanted arrows~$Q_{\lambda}\adj Q\adj Q_{\rho}$ in the recollement~\eqref{eq:sing-recollement}.
The left-hand vertical arrows $\cc_*\adj\cc^!$ are simply a shorthand for~\eqref{eq:Infl-adjoints}:
\[
\cc_*:=\Infl_1^G\qquadtext{and}\cc^!:=(-)^{\rmh G}.
\]
There are several reasons for this notation. First, it is lighter in formulas involving iterated compositions. Second, it evokes the algebro-geometric notation $\cc^*\adj \cc_*\adj \cc^!$ for an imaginary closed immersion $\cc\colon\Spec(\CR)\hook\Spec(\CRG)$ -- that actually makes sense if $G$ is abelian. (And we do have a left adjoint $\cc^*$ too, namely the left-derived functor of $G$-orbits~$(-)_{\rmh G}$.) Finally, it allows for a simple notation at the level of $\Komp(\Inj)$, namely the yet-to-be-explained~$\hat\cc_*\adj\hat\cc^!$ on the right-hand side of~\eqref{eq:main}.

For this, we apply \cite[\S\,6]{krause:stabX} to the \emph{exact} functor (denoted~$F$ in \loccit) $\Infl_1^G\colon\MMod{\CR}\to\MMod{\CRG}$. Its right adjoint~$(-)^G\colon \MMod{\CRG}\to \MMod{\CR}$ preserves injectives and our $\hat\cc^!\colon \Komp(\Inj(\CRG))\to \Komp(\Inj(\CR))$ is simply $(-)^G$ degreewise.
Its left adjoint~$\hat\cc_*\colon \Komp(\Inj(\CR))\to \Komp(\Inj(\CRG))$ is more subtle than just inflation. It is Krause's construction, namely $\hat\cc_*$ is defined as the composite
\[
\hat\cc_*\colon \Komp(\Inj(\CR))\xinto{\ J\ }\Komp(\MMod{\CR})\xto{\ \Infl_1^G\ }\Komp(\MMod{\CRG})\xonto{\ J_{\lambda}\ }\Komp(\Inj(\CRG)),
\]
where $J\colon \Komp(\Inj)\into\Komp(\Mod)$ is the inclusion and~$J_{\lambda}\colon \Komp(\Mod)\onto\Komp(\Inj)$ its left adjoint, as in~\eqref{eq:J-adjoints}. Using that $J$ is fully faithful, it is easy to see that $\hat\cc_*\adj\hat\cc^!$. (Although we had a derived left adjoint $\cc^*\adj\cc_*$ there is no $\hat\cc^*\adj\hat\cc_*$ on~$\Komp(\Inj)$.)

By \cite[Lemma~6.3]{krause:stabX}, since inflation is exact, we have
\begin{equation}
\label{eq:Q-c_*}%
Q\circ \hat \cc_*\cong\cc_*\circ Q\colon \Komp(\Inj(\CR))\to \Der(\CRG).
\end{equation}
From this we deduce, by taking right adjoints, that
\begin{equation}
\label{eq:c^!-Qr}%
\hat\cc^!\circ Q_{\rho}\cong Q_{\rho}\circ\cc^!.
\end{equation}

Note that since the functor $(-)^G\colon \MMod{\CRG}\to \MMod{\CR}$ preserves coproducts, so does the induced~$\hat\cc^!$ on $\Komp(\Inj)$. Thus its left adjoint preserves compacts:
\begin{equation}
\label{eq:hat-c_*-compacts}%
\hat \cc_*(\Komp(\Inj(\CR))^c)\subseteq \Komp(\Inj(\CRG))^c.
\end{equation}
\end{Cons}

\begin{Rem}
\label{Rem:R-action}%
On every $Y\in\Komp(\Inj(\CRG))$ the comonad $\hat\cc_*\hat\cc^!$ equals by construction
\[
\hat\cc_*\hat\cc^!(Y)=J_{\lambda}\,\Infl_1^G J\,\hat\cc^!(Y)=J_{\lambda}\Infl_1^G(J(Y)^{G})\cong J_{\lambda}\,\ihom_{\CRG}(\CR,Y)
\]
where $\ihom_{\CRG}(\CR,Y)$ has trivial $G$-action.
This leads us to bimodule actions:
\end{Rem}
\begin{Lem}
\label{lem:Db-acts-on-K_-(Inj)}%
There is an action of the bounded derived category of $\CR(G\times G\op)$-modules on~$\Komp_-(\Inj(\CRG))$, in the form of a well-defined bi-exact functor
\[
[-,-]\colon \Db(\CR(G\times G\op))\op\times \Komp_-(\Inj(\CRG))\to \Komp_-(\Inj(\CRG))
\]
given by the formula $[L,Y]=J_{\lambda}(\ihom_{\CRG}(L,Y))$.
\end{Lem}

\begin{proof}
At the level of module categories, there is an action
\[
\ihom_{\CRG}(-,-)\colon\MMod{\CR(G\times G\op)}\op\times\MMod{\CRG}\to\MMod{\CRG}
\]
which takes $(L,Y)$ to the abelian group~$\Hom_{\CRG}(L,Y)$ built by viewing~$L$ as a left $\CRG$-module via its left $G$-action, and then making the output $\Hom_{\CRG}(L,Y)$ into a left~$G$-module $\ihom_{\CRG}(L,Y)$ by using the `remaining' right $G$-action on~$L$. Being additive in both variables, this passes to a bi-exact functor on homotopy categories
\begin{equation}
\label{eq:bimodules-action-underived}%
\ihom_{\CRG}(-,-)\colon\Komp(\CR(G\times G\op))\op\times\Komp(\CRG)\to\Komp(\CRG)
\end{equation}
(by totalizing via $\prod$, which is irrelevant in our bounded case).
This yields
\[
[-,-]:\Kb(\CR(G\times G\op))\op\times\Komp_{-}(\Inj(\CRG))\xto{(\ref{eq:bimodules-action-underived})}\Komp_-(\CRG)\xto{J_{\lambda}}\Komp_-(\Inj(\CRG)).
\]
The preservation of left-boundedness by~$J_{\lambda}$ is \Cref{Lem:resolutions}\,\eqref{it:resolution-Jl}.
To show that this descends to the derived category in the first variable, let $L\in\Kb(\CR(G\times G\op))$ be acyclic and let $Y\in\Komp_{-}(\Inj(\CRG))$, and let us show that $J_{\lambda}(\ihom_{\CRG}(L,Y))=0$.
By \Cref{Lem:resolutions}\,\eqref{it:resolution-Jl} again, it suffices to show that $\ihom_{\CRG}(L,Y)$ is acyclic.
But $\Hm_n\ihom_{\CRG}(L,Y)=\Hom_{\Komp(\CRG)}(L[n],Y)$ vanishes since $Y$ is a left-bounded complex of injectives and $L$ is acyclic (as complex of $\CRG$-modules as well).
\end{proof}

\begin{Rem}
\label{Rem:RG-action}%
Each object $L$ in $\Db(\CR(G\times G\op))$ thus defines an exact endofunctor
\[
[L,-]\colon \Komp_-(\Inj(\CRG))\to \Komp_-(\Inj(\CRG)).
\]
For instance, $[\CRG,-]\cong\Id$ whereas $[\CR,-]\cong \hat\cc_*\hat\cc^!$ is our comonad, by \Cref{Rem:R-action}.
We use this to show that some~$Y\in\Komp_-(\Inj(\CRG))$ can be recovered from~$\hat\cc_*\hat\cc^!(Y)$.
\end{Rem}

\begin{Lem}
\label{Lem:conservative-p-gp}%
Let $G$ be a finite $p$-group and $Y\in\Komp_-(\Inj(\CRG))$ such that $p^n\cdot\id_{Y}=0$ for~$n\gg0$. Then $Y$ belongs to~$\thick(\hat\cc_*\hat\cc^!(Y))$ in $\Komp_-(\Inj(\CRG))$.
\end{Lem}
\begin{proof}
As explained in \Cref{Rem:RG-action}, we need to show that in~$\Komp_-(\Inj(\CRG))$
\[
[\CRG,Y]\in\thick([\CR,Y]).
\]
Since $p^n\cdot Y=0$, we also have $p^n\cdot[\CRG,Y]=0$ and the octahedron axiom gives (see~\cite[Remark~2.27]{balmer-gallauer:resol-small})
\[
[\CRG,Y]\in\thick(\cone([\CRG,Y]\xto{p}[\CRG,Y]))=\thick([\cone(\CRG\xto{p}\CRG),Y])
\]
using exactness of~$[-,Y]$.
Hence it suffices to prove in~$\Db(\CR(G\times G\op))$ that
\[
\cone(\CRG\xto{p}\CRG)\in\thick(\CR).
\]
This last statement is independent of~$Y$.
By scalar extension along~$\bbZ\to \CR$ (\Cref{Rec:scalar-extension}), it suffices to prove that in~$\Db(\bbZ(G\times G\op))$
\[
\cone(\bbZ G\xto{p}\bbZ G)\in\thick(\bbZ).
\]
Consider the exact functor $i_*\colon \Db(\Fp(G\times G\op))\to \Db(\bbZ(G\times G\op))$ obtained by restriction-of-scalars. The above $\cone(\bbZ G\xto{p}\bbZ G)$ is nothing but~$i_*(\Fp G)$, and $i_*(\Fp)\cong\cone(\bbZ\xto{p}\bbZ)$ belongs to~$\thick(\bbZ)$. So we are reduced to show that $\Fp G\in\thick(\Fp)$ in~$\Db(\Fp(G\times G\op))$, which is clear since $G\times G\op$ is also a $p$-group.
\end{proof}

\begin{Lem}
\label{Lem:main-prepa}%
Let $G$ be a $p$-group. Let $X\in\Db(\CRG)$ be $p$-torsion (\ie $p^n\cdot X=0$ for $n\gg 0$) and $G$-cohomologically perfect. Then $X$ belongs to~$\thick(\CR)$.
\end{Lem}

\begin{proof}
By \Cref{Lem:Ker-sing}, the assumption $0=\CInvGR(X)=I_\lambda\, Q_{\rho}\,\cc^!(X)$ implies that
\begin{equation}
\label{eq:aux-Qr(X)}%
Q_{\rho}\, \cc^!(X)\in\Loc{Q_{\rho}(\CR)}
\end{equation}
in $\Komp(\Inj(\CR))$. Applying to this relation the (coproduct-preserving) left adjoint $\hat\cc_*\colon\Komp(\Inj(\CR))\to \Komp(\Inj(\CRG))$ of~\eqref{eq:main}, we obtain in~$\Komp(\Inj(\CRG))$
\[
\hat \cc_*\, \hat\cc^!\, Q_{\rho}(X)\underset{\textrm{(\ref{eq:c^!-Qr})}}\cong\hat \cc_*\, Q_{\rho}\, \cc^!(X)\underset{\textrm{(\ref{eq:aux-Qr(X)})}}\in\Loc{\hat \cc_*Q_{\rho}(\CR)}.
\]
Hence by \Cref{Lem:conservative-p-gp} with $Y=Q_{\rho}(X)$, which is $p$-torsion since~$X$ is, we have
\begin{equation}
\label{eq:aux-main-1}%
Q_{\rho}(X)\in\thick(\hat\cc_*\hat\cc^!Q_{\rho}(X))\subseteq\Loc{\hat \cc_*Q_{\rho}(\CR)}
\end{equation}
in $\Komp(\Inj(\CRG))$. Now $Q_{\rho}(X)$ is compact in~$\Komp(\Inj(\CRG))$ by~\Cref{Rec:singularity}, and $\hat\cc_*Q_{\rho}(\CR)$ is compact because $Q_{\rho}(\CR)$ is and because $\hat\cc_*$ preserves compacts~\eqref{eq:hat-c_*-compacts}. So, by \cite[Lemma~2.2]{neeman:thomason-localization}, we can replace `Loc' by `thick' in~\eqref{eq:aux-main-1}, giving us the relation
\[
Q_{\rho}(X)\in\thick\big(\hat\cc_*Q_{\rho}(\CR)\big)
\]
in~$\Komp(\Inj(\CRG))$. Applying $Q\colon \Komp(\Inj(\CRG))\onto\Der(\CRG)$ and $QQ_{\rho}\cong\Id$, we get
\[
X\in \thick\big(Q \hat \cc_* Q_{\rho}(\CR)\big)\overset{\textrm{(\ref{eq:Q-c_*})}}{\ =\ }\thick\big(\cc_* Q Q_{\rho}(\CR)\big)= \thick\big(\cc_*(\CR)\big).
\]
Of course, $\cc_*(\CR)$ is just $\CR$ with trivial $G$-action, viewed in~$\Db(\CRG)$.
\end{proof}

\begin{Lem}
\label{Lem:main-p-gp}%
Let $G$ be a $p$-group and $X\in\Db(\CRG)$. The following are equivalent:
\begin{enumerate}[\rm(i)]
\item
\label{it:main-X-in-A(G)}%
$X\in\DpermGR$; see~\eqref{eq:Dperm}.
\item
\label{it:main-X-coh-perf}%
$X$ is cohomologically perfect (\Cref{Def:coh-inv}).
\item
\label{it:main-X-G/1-coh-perf}%
$X$ is $G$-cohomologically perfect and $\CR$-perfect (\Cref{Rec:R-perfect}).
\end{enumerate}
\end{Lem}

\begin{proof}
The implication \eqref{it:main-X-in-A(G)}$\Rightarrow$\eqref{it:main-X-coh-perf} is \Cref{Prop:chi-necessary}, and
the implication \eqref{it:main-X-coh-perf}$\Rightarrow$\eqref{it:main-X-G/1-coh-perf} is trivial by \Cref{Def:coh-inv}.
For the implication \eqref{it:main-X-G/1-coh-perf}$\Rightarrow$\eqref{it:main-X-in-A(G)} suppose that $\CInvGR(X)=0$ and $X$ is $\CR$-perfect.
By \cite[Corollary~2.26]{balmer-gallauer:resol-small}, there exists an exact triangle in~$\Db(\CRG)$
\[
P\to X \oplus \Sigma X \to T\to\Sigma P
\]
where $P$ is a bounded complex of permutation modules (and therefore belongs to~$\DpermGR$)
and where $T\in\Db(\CRG)$ is $p$-torsion.
Since~$P$ and~$X$ are $G$-cohomologically perfect so is~$T$.
Then \Cref{Lem:main-prepa} tells us that $T\in\thick(\CR)$ in~$\Db(\CRG)$. As $\CR\in\DpermGR$ we get $X\in\DpermGR$ as well.
\end{proof}

\begin{Rem}
\label{Rem:coh-perfect}%
For $G$ a $p$-group the equivalence \eqref{it:main-X-coh-perf}$\Leftrightarrow$\eqref{it:main-X-G/1-coh-perf}
in~\Cref{Lem:main-p-gp} shows that $G$-cohomological perfection together with $\CR$-perfection does imply $H$-cohomological perfection for all~$H\le G$. This is sharp by \Cref{Exa:coh-perfect-not-1-perfect} and \Cref{Exa:1-perf-not-A(G)}.
\end{Rem}

Here is the main result. The category $\DpermGR=\Img(\bar{\Permtomod})$ can be found in~\eqref{eq:Dperm} and in \Cref{Rec:image-Fbar}. The invariant~$\CInvHR$ is in \Cref{Def:coh-inv}.

\begin{Thm}
\label{Thm:main}%
Let $G$ be a finite group and~$\CR$ a commutative noetherian ring.
Let $X\in\Db(\CRG)$ be a bounded complex. The following properties of~$X$ are equivalent:
\begin{enumerate}[\rm(i)]
\item
\label{it:Acat}%
The complex $X$ belongs to~$\DpermGR$.
\smallbreak
\item
\label{it:coh-perf}%
It is cohomologically perfect: $\CInvHR(X)=0$ in~$\Dsing(\CR)$ for all subgroups~$H\le G$.
\smallbreak
\item
\label{it:coh-perf-Sylow}%
It is $\CR$-perfect, \ie the underlying complex $\Res^G_1(X)\in\Db(\CR)$ is perfect, and it is $H$-cohomologically perfect, $\CInvHR(X)=0$, for every Sylow subgroup~$H\le G$.
\end{enumerate}
\end{Thm}

\begin{proof}
Implication (\ref{it:Acat})$\Rightarrow$(\ref{it:coh-perf}) is \Cref{Prop:chi-necessary}.
The implication (\ref{it:coh-perf})$\Rightarrow$(\ref{it:coh-perf-Sylow}) is trivial (\Cref{Def:coh-inv}).
If we assume~\eqref{it:coh-perf-Sylow}, then \Cref{Lem:main-p-gp} implies that $\Res^G_H(X)\in\Dperm{H}{\CR}$ for every Sylow subgroup~$H\le G$.
Since the indices of all Sylow subgroups are coprime, it is easy to deduce that $X\in\DpermGR$, see \cite[Corollary~2.21]{balmer-gallauer:resol-small}.
So the three conditions \eqref{it:Acat}$\Leftrightarrow$\eqref{it:coh-perf}$\Leftrightarrow$\eqref{it:coh-perf-Sylow} are equivalent.
\end{proof}

\begin{Rem}
As in \Cref{Rec:R-perfect}, it suffices to test~\eqref{it:coh-perf-Sylow} for the $p$-Sylow subgroups $H$ corresponding to primes~$p$ that are non-invertible on~$X$ (and in~$\CR$).

One can also replace~\eqref{it:coh-perf-Sylow} by only asking $X$ to be $E$-cohomologically perfect for every elementary abelian $p$-subgroup~$E\le G$. See~\Cref{Rec:ImF-stability}.
\end{Rem}

\begin{Rem}
\Cref{Thm:main-intro} of the Introduction follows from \Cref{Thm:main}.
\end{Rem}

\begin{Cor}
\label{Cor:coh-ring-mod}%
With $X\in\Db(\CRG)$ as in \Cref{Thm:main}, the conditions~\eqref{it:Acat}, \eqref{it:coh-perf}, \eqref{it:coh-perf-Sylow} are also equivalent to:
\begin{enumerate}[\rm(i)]
\setcounter{enumi}{3}
\item
\label{it:c^!(X)-thick-c^!(R)}%
In $\Der(\CR)$, we have $X^{\rmh H}\in\thick(\SET{\CR^{\rmh K}}{K\le G})$ for every subgroup~$H\le G$.
\end{enumerate}
\end{Cor}
\begin{proof}
Let us denote by $\cat{J}:=\thick(\SET{\CR^{\rmh K}}{K\le G})$ the thick subcategory of~$\Der(\CR)$ appearing in~\eqref{it:c^!(X)-thick-c^!(R)}. For~\eqref{it:c^!(X)-thick-c^!(R)}$\To$\eqref{it:coh-perf}, note that $\StabD(\CR^{\rmh K})=\CInv{K}{\CR}(\CR)=0$ by \Cref{Lem:chi-necessary}. So $\cat{J}\subseteq\Ker(\StabD)$ and therefore $X^{\rmh H}\in\cat{J}$ implies $\CInv{H}{\CR}(X)=\StabD(X^{\rmh H})=0$. For~\eqref{it:Acat}$\To$\eqref{it:c^!(X)-thick-c^!(R)}, it is sufficient to prove that for all subgroups~$H,L\le G$ we have $(\CR(G/L))^{\rmh H}\in\cat{J}$. This follows from the Mackey formula and \Cref{Prop:CInv-Ind}.
\end{proof}

\begin{Rem}
\label{Rem:coh-ring-mod}%
The inflation functor $\cc_*:\Der(\CR)\to\Der(\CRG)$ is monoidal and its right adjoint $\cc^!=(-)^{\rmh G}:\Der(\CRG)\to\Der(\CR)$ is therefore lax monoidal.
In particular, $\cc^!\cc_*(\unit)=\CR^{\rmh G}$ is a ring object, namely the `cohomology ring' of $G$ with coefficients in $\CR$, and every object $X\in\Der(\CRG)$ gives rise to a module $X^{\rmh G}$ over this ring.

With this in mind, and the fact that for every ring~$\Lambda$ we have~$\Dperf(\Lambda)=\thick(\Lambda)$, the terminology `cohomologically perfect' of \Cref{Def:coh-inv} is somewhat justified by the equivalent formulation given in part~\eqref{it:c^!(X)-thick-c^!(R)} of \Cref{Cor:coh-ring-mod}.
\end{Rem}

\begin{Rem}
\label{Rem:DPerm}%
Neeman's Localization Theorem~\cite[Theorem~2.1]{neeman:thomason-localization} suggests that the equivalence~\eqref{eq:Kb/Kbac=Dperm} is the compact tip of an iceberg.
To describe this iceberg, we define the \emph{(big) derived category of permutation modules}
\[
\DPermGR:=\Loc{\perm(G;\CR)}=\Loc{\SET{\CR(G/H)}{H\le G}}
\]
as the localizing subcategory of $\Komp(\MMod{\CRG})$ generated by permutation modules.
Its compact part is precisely $\Kb(\permGR^\natural)$.
It follows from \Cref{Lem:resolutions}\,\eqref{it:resolution-Qr=Jl} that $J_{\lambda}(\CR(G/H))\cong Q_{\rho}(\CR(G/H))$ so that we obtain a coproduct-preserving and compact-preserving exact functor
\begin{equation}
\label{eq:Fplus}%
\Fplus:=(J_{\lambda})\restr{\DPermGR}\colon \DPermGR \to \ACAT{G}{\CR},
\end{equation}
where the latter is the localizing subcategory of $\Komp(\Inj(\CRG))$ generated by the $Q_{\rho}(\CR(G/H))$, $H\leq G$.
This~$\Fplus$ is a finite localization which extends beyond compacts the canonical functor~$\Permtomod$ of~\eqref{eq:Kb->Db}.
In particular, it induces an equivalence
\begin{equation*}
\label{eq:bar-F-plus}%
\barFplus:\frac{\DPermGR}{\Loc{\Kbac(\perm(G;\CR))}}\isoto\ACAT{G}{\CR}.
\end{equation*}
On compacts, this equivalence identifies with the equivalence~$\bar{\Permtomod}$ of~\eqref{eq:Kb/Kbac=Dperm}.
Note that if $\CR$ is regular, \eqref{eq:Fplus} exhibits Krause's homotopy category of injectives $\Komp(\Inj(\CRG))$ as a finite localization of~$\DPermGR$. (\footnote{\,We mention as a curiosity that for $\CR=\kk$ a field (or just self-injective) we have the inclusion $\Komp(\Inj(\kk G))\subseteq\DPerm(G;\kk)$ as localizing subcategories of $\Komp(\MMod{\CRG})$, and the finite localization $\Fplus$ is simply the left adjoint to that inclusion.})

Unfortunately, for general~$\CR$, we are unable to characterize the subcategory $\ACAT{G}{\CR}\subseteq\Komp(\Inj(\CRG))$ along the lines of \Cref{Thm:main}.
An immediate extension of that result is blocked because the cohomological singularity functor is not coproduct-preserving.
\end{Rem}

\begin{Rem}
The above definition of~$\DPermGR$ does not do justice to the big derived category of permutation modules.
We refer to the expository note~\cite{balmer-gallauer:Dperm} for a more conceptual approach.
There we also explain that $\DPerm(G;\CR)$ is equivalent to the derived category of cohomological $\CR$-linear Mackey functors on~$G$ and, after suitably extending to profinite groups, to the triangulated category of Artin motives over a field with absolute Galois group~$G$ and with coefficients in~$\CR$, in the sense of Voevodsky.
\end{Rem}



\end{document}